\documentclass[draft]{amsart}

\usepackage{comment}
\usepackage{amssymb,amscd}
\usepackage[all]{xy}
\numberwithin{equation}{section}

\newtheorem{theorem}{Theorem}[section]
\newtheorem{corollary}[theorem]{Corollary}
\newtheorem{proposition}[theorem]{Proposition}
\newtheorem{lemma}[theorem]{Lemma}

\newtheorem{definition}[theorem]{Definition}

\newtheorem{remark}[theorem]{Remark}

\newcommand{\la}{\langle}
\newcommand{\ra}{\rangle}
\newcommand{\im}{{\rm Im}}

\renewcommand{\L}{\mathcal{L}}

\newcommand{\F}{\mathcal{F}}
\newcommand{\J}{\mathcal{J}}

\newcommand{\R}{\mathbb{R}}
\newcommand{\C}{\mathbb{C}}
\newcommand{\Z}{\mathbb{Z}}

\newcommand\lie[1]{\mathfrak{#1}}
\newcommand{\Diff}{\mathfrak{Diff}}
\newcommand{\Con}{\mathfrak{Con}}
\newcommand{\con}{\lie{con}}

\newcommand{\fk}{\lie{k}}
\newcommand{\ft}{\lie{t}}

\def \inv{^{-1}}
\def \i{\sqrt{-1}}
\def \Sp{\mathrm{Sp}}

\begin{document}

\title{The cutting construction of toric symplectic and contact manifolds}

\author{Yushi Okitsu}
\address{Department of Mathematics, Tokyo Institute of Technology, 2-12-1, O-okayama, Meguro, Tokyo 152-8551, Japan}
\email{okitsu.y.aa@m.titech.ac.jp}

\thanks{}
\date{\today}

\begin{abstract} 
We introduce the cutting construction of possibly non-compact symplectic toric manifolds, in particular, toric symplectic cones that correspond to a weakly convex good cone. Since the symplectization of a toric contact manifold is a toric symplectic cone, we can also construct toric contact manifolds that correspond to a weakly convex good cone by cutting construction. (Note that these toric contact manifolds can not be constructed by Delzant construction.) We further prove there are no toric Sasakian structures on these contact manifolds. From this, contact toric manifolds of toric K-contact type are of toric Sasakian type.  
\end{abstract}

\maketitle
\markboth{The cutting construction of toric symplectic and contact manifolds}{YUSHI OKITSU}

\section{Introduction}
The history of the classification of symplectic toric manifolds began with the classification theorem of compact symplectic toric manifolds by T.Delzant \cite{D}. Thereafter, E.Lerman showed the classification theorem of compact connected contact toric (c.c.c.t. for short) manifolds \cite{L2}, and Y.Karshon and E.Lerman showed the classification theorem of non-compact symplectic toric manifolds \cite{KL}. The moment map on symplectic toric manifolds plays an important role in these classification theorems. Roughly speaking, there is a one-to-one correspondence between symplectic toric manifolds and certain convex sets except free action case. Hence it is important to construct symplectic toric manifolds by convex sets. After Delzant's work, such a construction is called the Delzant construction and the corresponding convex sets are called the Delzant polytopes (c.f. subsection \ref{Del}). This construction is only applicable to strongly convex sets. Hence we need the cutting construction to construct a symplectic toric manifold which corresponds to a weakly convex set.

In this paper, we define a new class of convex sets called unimodular sets (c.f. Definition \ref{uni}) and introduce the cutting construction (c.f. Theorem \ref{main}). As an application of cutting construction, we construct a c.c.c.t. manifold which corresponds to a weakly convex cone (c.f. Definition \ref{weak}). These manifolds were not constructed in E. Lerman's paper \cite{L2}, which has been pointed out in \cite{Y}. Moreover we can also construct an arbitrary  c.c.c.t. manifold which corresponds to a strongly convex cone by using the cutting construction. Therefore we are able to construct an arbitrary c.c.c.t. manifold with non-free toric action:

 \begin{theorem}\label{contact}
Let $\Delta := \{x \in (\R^n)^*\setminus\{0\}\;|\;\langle x,\eta_i\rangle \geqq 0, i=1,\cdots,N\}$ be a unimodular cone and $(S^*(T^n) \times \C^N ,\sum_{i=1}^n x_i d\theta_i + \sqrt{-1}\sum_{i=1} ^N{(z_id\overline{z}_i-\overline{z_i}dz_i)},T^n \times T^N,\Psi)$ be a toric contact manifold described as follows:\\
$T^n \times T^N$-action on $S^*(T^n) \times \C^N$ is given by
\begin{equation}
(s_1,\cdots, s_n,t_1,\cdots,t_N)\cdot(x,\theta,z_1,\cdots,z_N)=(x,\theta+\sum_{i=1}^n{s_ie_i}+\sum_{i=1} ^N{t_i\eta_i},e^{-\i t_1} z_1,\cdots, e^{-\i t_N} z_N)
\end{equation}
where $(x,\theta)$ is restriction of action-angle coordinates to the co-sphere bundle $S^*(T^n)\cong S^{n-1} \times \R^n/2\pi\Z^n$\\$\subset \R^n\times T^n$ of the torus $T^n$, and $e_i$ is the standard basis of $\R^n$.
Moreover the moment map of this action is described as
\begin{equation}
\Psi:S^*(T^n)\times \C^N \to (\R^n\oplus\R^N)^*, (x,\theta,z_1,\cdots,z_N) \mapsto x \oplus (\la x,\eta_1\ra-\|z_1\|^2,\cdots,\la x,\eta_N \ra-\|z_N\|^2).
\end{equation}
We denote the $(\R^N)^*$-part of $\Psi$ by $\mu$. If we take the quotient $M_\Delta := \mu\inv(0)/T^N$, then there are an induced contact structure $\xi_\Delta$, an induced $T^n$-action and an induced moment map $\Phi_\Delta$ on $M_\Delta$ such that the moment cone $C(\Psi_\Delta)$ coincides with $\Delta\cup\{0\}$ where $\Psi_\Delta$ is the universal moment map of $(M_\Delta,\xi_\Delta,T^n,\Phi_\Delta)$.
\end{theorem}

We further prove there do not exist toric K-contact  structures on these manifolds:

\begin{theorem}\label{main2}
Let $(M, \xi , T, \Phi)$ be a toric contact manifold which corresponds to a weakly convex cone. Then, a toric contact metric manifold $(M, \alpha,T,\Phi, g)$ is not a toric K-contact manifold for any $T^n$- invariant contact form $\alpha$, metric g, and almost CR structure $\Phi$.
\end{theorem}

This paper is organized as follows.
In section 2, we introduce cutting construction.
In section 3, we observe the diffeomorphism types of manifolds which correspond to weakly convex good cones.
In section 4, we apply cutting construction to contact manifolds and obtain c.c.c.t. manifolds. In particular, we get c.c.c.t. manifolds that correspond to weakly convex good cones, and further prove there do not exist toric K-contact structures on these manifolds.
In Appendix, we interpret the cutting construction in term of the K\"{a}hler geometry. Moreover we compute the canonical K\"{a}hler structure of cutting constructed manifolds and its symplectic potential explicitly.

\subsection*{Acknowledgement}
I would like to express my deepest gratitude to Prof. A. Futaki whose comments and suggestions were of inestimable value for my study.

\section{Preliminaries}
\subsection{Basic facts and basic notations.}
Let $(M,\omega)$ be a symplectic manifold with symplectic form $\omega$ and suppose $G$ is a Lie group with Lie algebra $\lie{g}$. Let $\Psi:G \rightarrow {\rm Sympl}(M,\omega)$ be a symplectic action of $G$ where ${\rm Sympl}(M,\omega)$ is the group of  symplectomorphisms that map $M$ to itself.
The action $\Psi$ is a Hamiltonian action if there exists a map $\mu:M \rightarrow \lie{g}^*$ satisfing:

(1) For each $X \in \lie{g}$, let $X_M$ be the vector field on M induced by the one-parameter subgroup $\{\exp(tX)\;|\;t\in\R\}\subset G$, then
\begin{equation}
\iota_{X_M}\omega = - d\langle \mu,X \rangle
\end{equation}
where $\iota$ is the interior product operator, and $\langle \cdot, \cdot \rangle$ is the algebraic pairing.

(2) $\mu$ is equivariant with respect to the given action $\Psi$ of $G$ on $M$ and the coadjoint action ${\rm ad}^*$ of $G$ on $\lie{g}^*$:
\begin{equation}
\mu \circ \Psi_g = {\rm ad}_g^* \circ \mu \;\;({}^\forall g \in G\,).
\end{equation}
The quadruple $(M,\omega,G,\mu)$ is then called a Hamiltonian $G$-space and $\mu$ is the moment map.

Suppose $G$ is commutative. Then since the coadjoint action is trivial, the above equivariance becomes invariance. For $g \in G$ and $p \in M$ we denote $\Psi_g (p)$ by $g \cdot p$.

Now we define a symplectic toric manifold to be a connected symplectic manifold $(M^{2n},\omega)$ equipped with an effective Hamiltonian action of a torus $T^n$ and with a choice of a corresponding moment map $\mu$.

In this paper we confine ourselves to the case where $G$ is a torus $T^n$, i.e. we only deal with toric cases. Hence we identify the Lie algebra of torus $G = T^n$ with $\R^n$ for the sake of convenience.
It is well known that a toric symplecic manifold corresponds to a certain convex polyhedral set.

Recall that a convex polyhedral set in $(\R^n)^* = \lie{g}^*$ is 
\begin{equation}
\Pi = \{x \in (\R^n)^*\;|\;\langle x,\eta_i\rangle \leqq \kappa_i, i=1,\cdots,N\}
\end{equation}
where $\eta_i \in \R^n$ (this is called outward conormal vector of the $i$-th facet) and $\kappa_i \in \R$.
\\~\\
We give the definition of a suitable class of convex polyhedral sets to describe the cutting construction.

\begin{definition}[unimodular set]\label{uni} Let $\Pi$ be as above.
A unimodular set $\Delta$ is  a relatively open subset of $\Pi$ satisfying:\\
$(1)$ if $\Delta$ has vertices then each vertex is simple, i.e. there are $n$ edges meeting at each vertex;\\
$(2)$ each outward conormal vector of a facet of $\Delta$ is in $\Z^n$ and primitive;\\
$(3)$ for any subset $I \subset \{1,\cdots,N\}$, the following holds: if $\Delta \cap F_I \neq \emptyset$ then $\{\eta_i\}_{i \in I}$ is a basis of integral lattice of a subtorus $K \subset T^n$, where $F_I :=\{x \in \Pi\;|\; \langle x,\eta_i \rangle = \kappa_i, i \in I\}$.
\end{definition}
\noindent
Moreover we define some convexities to state the main result.

\begin{definition}[strongly convex and weakly convex] \label{weak}
Let $\Pi$ be an $n$-dimensional convex polyhedral set with outward conormal vectors $\eta_1,\cdots,\eta_N$. Then $\Pi$ is strongly convex if and only if its conormal vectors span the whole space, i.e. $\R$-$span\{\eta_1,\cdots,\eta_N\} = (\R^n)^*$. Moreover a unimodular set $\Delta=U \cap \Pi$ is strongly convex if and only if $\Pi$ is strongly convex, where $U$ is an open subset of $(\R^n)^*$. A unimodular set is weakly convex if and only if it is not strongly convex. 
\end{definition}
\noindent
Note that the strongly or weakly convexity of a unimodular set $\Delta=U \cap \Pi$ depend on the ambient polyhedral set $\Pi$, hence we should consider a unimodular set $\Delta=U \cap \Pi$ coincides with another one $\Delta^\prime = U^\prime\cap\Pi^\prime$ if and only if  $\Delta$ coincides with $\Delta^\prime$ as a set, and $\Pi$ coincides with $\Pi^\prime$ as convex polyhedral set.


The source of the word \textit{unimodular} is the paper of Y.Karshon and E.Lerman$ ($\cite{KL}$)$.
If a unimodular set $\Pi=\Delta$ is a polytope, i.e. it is a \textit{compact} polyhedral set, then it is called a Delzant polytope,
 and the condition is referred to as simple, (2) to as rational, and (3) to as smooth. Notice that a Delzant polytope is necessarily strongly convex. 
 If a unimodular set is a cone without the conical point then it is a good cone without the conical point (cf. Definition \ref{good}). That is the moment image of a symplectic cone. Note that if a cone that corresponds to a symplectic cone is weakly convex, we define that the conical point is the origin.

For a Delzant polytope, we have the following well-known result by T.Delzant \cite{D}:

\begin{theorem}[Delzant]\label{Delzant}
Compact connected symplectic toric manifolds are classified by Delzant polytopes.
More precisely, there is a one-to-one correspondence between $n$-dimensional Delzant polytopes and $2n$-dimensional compact symplectic manifolds up to $T^n$-equivariant symplectomorphisms that preserve a moment map.
\end{theorem}

\begin{remark}
This theorem is generalized to a non-compact and disconnected case by Y.Karshon and E.Lerman \cite{KL}. Their result says there is a similar correspondence between unimodular sets and symplectic toric manifolds. 

In general, the moment image of a symplectic toric manifold is not a unimodular set and symplectic toric manidfolds can not be classified by its moment images $($see \cite{KL}$)$.  
\end{remark}

Recall the Denlzant construction which is in common use to prove the existence part in Theorem \ref{Delzant}.

\subsection{Delzant construction}\label{Del}
Suppose we are given a Delzant polytope $\Delta = \{x \in (\R^n)^*\;|\;\langle x,\eta_i\rangle \leqq \kappa_i, i=1,\cdots,N\}$.
Let $\{e_1,\cdots,e_N\}$ denote the standard basis of $\R^N$. Consider the map $\pi : \R^N \rightarrow \R^n = \lie{g}$ given by $\pi(\sum a_ie_i)=-\sum a_i \eta_i$. Since $\Delta$ is strongly  convex, $\pi$ is surjective. Moreover by smoothness, $\pi$ maps $\Z^N$ onto $\Z^n$. Hence $\pi$ induces the surjective map $\tilde{\pi}:T^N = \R^N/(2\pi \Z)^N \rightarrow \R^n/(2\pi \Z)^n = T^n$ and an injective dual map $\pi^*: (\R^n)^* \rightarrow (\R^N)^*$.
Let $K$ be the kernel of $\tilde{\pi}$, and $\lie{k}$ the Lie algebra of $K$. Now $K = \lie{k}/(2\pi\Z)^N$ holds by smoothness of $\Delta$ and so $K$ is connected. Then we have the following three exact sequences:
\begin{equation}
\begin{CD}
0@>>>\lie{k}@>i>>\R^N@>\pi>>\R^n@>>>0,
\end{CD}
\end{equation}

\begin{equation}\label{exa}
\begin{CD}
0@>>>K@>i>>T^N@>\tilde{\pi}>>T^n@>>>0,
\end{CD}
\end{equation}

\begin{equation}
\begin{CD}
0@>>>(\R^n)^*@>\pi^*>>(\R^N)^*@>i^*>>\lie{k}^*@>>>0,
\end{CD}
\end{equation}
where $i$ is the inclusion map.

Now consider $\C^N$ with symplectic form $\sqrt{-1}\sum dz_i \wedge d\overline{z_i}$, and standard Hamiltonian action of $T^N$ given for $t=(t_1,\cdots,t_N)\in (\R/2\pi \Z)^N$ and $z =(z_1,\cdots,z_N)$ by $t \cdot z = (e^{t_1\sqrt{-1}}z_1,\cdots, e^{t_N \sqrt{-1}}z_N)$. Then we can take the moment map of this action as follows:
\begin{equation*}
\sigma:\C^N \rightarrow (\R^N)^*, z \mapsto (\|z_1\|^2,\cdots,\|z_N\|^2) - (\kappa_1,\cdots,\kappa_N)
\end{equation*}
The action of $K$ on $\C^N$ is induced by the restriction of the action of $T^N$. Moreover the moment map of this $K$-action is $i^* \circ \sigma$.

Let $Z = (i^*\circ \sigma)\inv(0)$ be the zero-level set. In fact, $Z$ is compact and $K$ acts on $Z$ freely. Then we get a compact connected $2n$-dimensional symplectic manifold $M_\Delta :=Z/K$ by the symplectic reduction. Let $\omega_\Delta$ be the reduced symplectic form. Moreover $M_\Delta$ is toric and its moment image is $\Delta$. For more details on this see \cite{C}.

\begin{remark}
We can also apply Delzant construction to strongly convex good cones $($cf.\cite{L2}$)$. But we can not apply this construction to weakly convex unimodular sets, especially weakly convex good cones. To construct corresponding manifolds in weakly convex cases, we use symplectic cuts, which we explain below.
\end{remark}

\subsection{Symplectic cuts}
Next, we review symplectic cuts due to E. Lerman \cite{L}.

\begin{definition}[symplectic cuts]
Let ($M$,$\omega$) be a symplectic manifold with symplectic form $\omega$ and suppose that the circle $S^1$ acts on $M$ in the Hamiltonian way with moment map $\phi : M \rightarrow \R$.
Now we take a symplectic manifold $(M \times \C,\omega + \sqrt{-1} dz \wedge d\overline{z})$ with the diagonal $S^1$-action. The moment map of this $S^1$-action is $\mu = \phi + \|z\|^2$.
If $S^1$ acts freely on $\phi \inv (\kappa)$, then $\mu\inv(\kappa)/S^1$ is nonsingular and becomes a symplectic manifold. We call this the \textbf{symplectic cut of M with respect to the ray $(-\infty,\kappa]$} and denote this by $M_{cut}^\kappa$.
\end{definition}

Here we considered $S^1$ as $\R/2\pi \Z$ so that the moment map of standard $S^1$-action on $\C$ is $\|z\|^2$.

\begin{remark}\label{ab}
Since $\mu \inv (\kappa) = \{(m,z)\in M \times \C\;|\;\kappa - \phi(m) = \|z\|^2>0 \} \sqcup \{(m,0) \in M \times \C \;|\; \phi(m) = \kappa \}$ and both parts are $S^1$-invariant, $M_{cut}^\kappa = \mu \inv (\kappa) / S^1 $ is the disjoint union of the quotient of these two parts. If we put $M_\kappa :=\{(m,0) \in M \times \C \;|\; \phi(m) = \kappa \} / S^1$ and $M_0 ^\kappa :=\{(m,z)\in M \times \C\;|\;\kappa-\phi (m) = \|z\|^2 >0 \}/S^1$ then we have $M_{cut} ^\kappa = M_\kappa \sqcup M_0 ^\kappa$.

Now consider the map 

\begin{equation}\label{interior}
\sigma :\phi \inv ((-\infty,\kappa))\rightarrow M_0 ^\kappa, m \mapsto [m,\sqrt{\kappa - \phi(m)}]
\end{equation}
where $[m,z]$ is the equivalence class of $(m,z) \in \mu \inv (\kappa)$.\\
One can easily see that $\sigma$ is a symplectomorphism (see \cite{L1}, Theorem 2.5), so $M_0 ^\kappa$ can be identified with $\phi \inv (-\infty,\kappa)$, hence we can consider that $M_0^\kappa$ is an open symplectic submanifold of $M$ in that sense. Nevertheless, in general, $\sigma$ is not a K\"{a}hler isometry, that is, $M_0^\kappa$ is not embedded in $M$ as a K\"{a}hler submanifold by $\sigma$.
\end{remark}

\subsection{Construction with symplectic cuts}
In this subsection, we construct a symplectic toric manifold which corresponds to $\Delta$, where $\Delta$ is a unimodular set. Here we are not assuming $\Delta$ is compact, and thus it is not necessarily a Delzant polytope. For simplicity, we assume that $\Delta$ is a polyhedral set $\{x \in (\R^n)^*\;|\;\langle x,\eta_i\rangle \leqq \kappa_i, i=1,\cdots,N\}$.\\

\noindent
{\bf Step $\bf 0$}: Let $M$ be a cotangent bundle of $n$-torus. Give the standard coordinate $(x_1,\cdots,x_n, \theta_1,\cdots,\theta_n)$ on $M \cong \R^n \times T^n$ and let $\omega$ be the standard symplectic form $\sum dx_i \wedge d\theta_i$ on $M$. Moreover we consider the canonical $T^n$-action on $M$, that is,
\begin{equation*}
(t_1,\cdots,t_n)\cdot(x_1,\cdots,x_n, \theta_1,\cdots,\theta_n) = (x_1,\cdots,x_n, \theta_1+ t_1,\cdots,\theta_n + t_n)
\end{equation*}
and take a moment map with this action as follows:
\begin{equation*}
\Phi : M \rightarrow (\R^n)^*, (x,\theta)\mapsto x,
\end{equation*}
where $x = (x_1,\cdots,x_n),\theta = (\theta_1,\cdots,\theta _n)$. Indeed, $\im\Phi=(\R^n)^*$.\\

\noindent
{\bf Step $\bf1$}: We construct the symplectic toric manifold which corresponds to $\{x \in (\R^n)^*\;|\;\langle x,\eta_1\rangle \leqq \kappa_1\}$.\\
We consider $(M \times \C, \omega + \sqrt{-1} dz\wedge d\overline{z})$ and the diagonal $S^1$-action on this, where the $S^1$-action on $M$ is the action of $\R$-${\rm span} \{\eta_1\}/2\pi\Z \subset T^n$ and the $S^1$-action on $\C$ is the standard action of $S^1\subset \C$. We take a moment map $\mu_1$ of this action as follows:
\begin{equation}
\mu_1:M \times \C \rightarrow \R^*, ((x,\theta),z)\mapsto \langle x, \eta_1 \rangle + \|z\|^2.
\end{equation}
Since the $S^1$-action is free on $\la\Phi,\eta_1\ra \inv (\kappa_1)$, we can take a symplectic cut $M_1 := \mu_1\inv(\kappa_1)/S^1$. Now we extend the $T^n$-action on $M$ to $M\times \C$ as the trivial action on the second factor $\C$ and take the product $T^n \times S^1$-action on $M\times \C$. The moment map $\Psi$ of this product action is 
\begin{equation}
\Psi = \Phi \oplus \mu_1:M \times \C \rightarrow (\R^n)^* \oplus \R^*, (x,\theta,z)\mapsto(x,\la x,\eta_1 \ra + \|z\|^2).
\end{equation}
Since $\Phi$ is $S^1$-invariant, we have the moment map $\Phi_1$ with $T^n$-action on $M_1$ which is induced by $T^n$-action on $M \times \C$ and the following diagram commutes:
\begin{center}
$\;$
\begin{xy}
(0,12) *{\mu_1 \inv (\kappa_1)} = "A", (20,12) *{M\times \C} ="B", (45,12) *{(\R^n)^*\oplus \R^*} ="C",
(0,0) *{M_1} ="D", (20,0) *{(\R^n)^*} ="E",
\ar^{\rm inc} "A";"B"
\ar^{\Psi} "B";"C"
\ar_{\rm quot} "A";"D"
\ar^{\Phi\oplus 0} "B";"E"
\ar^{\rm proj} "C";"E"
\ar_{\exists \Phi_1} "D";"E"
\end{xy}
\end{center}

To compute $\Phi_1(M_1)$ we remark the following:
\begin{itemize}
\item $\Phi_1(M_1)=\Phi_1 \circ ({\rm quot})(\mu\inv(\kappa_1))=({\rm proj})\circ\Psi(\mu_1\inv(\kappa_1)) = (\Phi \oplus 0)(\mu_1\inv(\kappa_1))$,
\item $\mu_1\inv(\kappa_1)=\{(x,\theta,z)\;|\;\la x,\eta_1\ra + \|z\|^2 = \kappa_1 \} = \{(x,\theta,z)\;|\;\la x,\eta_1\ra \leqq \kappa_1 , \|z\|=\sqrt{\kappa_1-\la x,\eta_1\ra}\}$.
\end{itemize}
Combine these remarks, one can see that $\Phi_1(M_1)$ coincides with $\{x \in (\R^n)^*\;|\;\langle x,\eta_1\rangle \leqq \kappa_1\}$. Moreover we have a symplectic toric manifold $(M_1,\omega_1,\phi_1)$, where $\omega_1$ is a reduced symplectic form which is induced by $\omega$. One can easily see that a point in the inverse image $\Phi_1\inv(\{\langle x,\eta_1\rangle=\kappa_1\})$ is fixed by the subgroup  $S_1 := \R$-${\rm span}\{\eta_1\}/2\pi\Z$ in $T^n$. In below, we denote each subgroups of this form: $\R$-${\rm span}\{\eta_i\}/2\pi\Z$ in $T^n$ by $S_i$.\\

\noindent
{\bf Step $\bf 2,\cdots,N$}: We repeat the cutting as above for $S_2, \cdots S_N$, then we have the symplectic toric manifold $(M_N,\omega_N,\Phi_N)$ with moment image $\Delta$. This is the end of the cutting construction.\\

\begin{remark}
Let $(M_k,\omega_k,\mu_k)$ be a symplectic toric manifold (orbifold) resulting from the k-times cuttings.
Note that when we cut $(M_k,\omega_k,\mu_k)$ by the $S_{k+1}$-action, $(M_k)_{\kappa_{k+1}}$ is non-singular if and only if the $S_{k+1}$-action is free on $\la \Phi_k,\eta_{k+1}\ra \inv (\kappa_{k+1})$.

Hence to be convinced that we can get a manifold at the end of this construction, we must prove $S_{k+1}$-action is free on $\Phi_k \inv(\Delta \cap F_{k+1})$, where $F_{k+1}$ is the $(k+1)$-th facet of $\Pi (= \Delta)$.
Let $p$ be a point in $\Phi_k \inv (\Delta \cap F_{k+1})$ and we take a maximal subset $I\in\{J\subset\{1,\cdots,k\}\;|\;\Phi_k(p)\in F_J\}$, where $F_J$ is as given in Definition \ref{uni} and where maximality is given by the inclusion property. Then $p \in \Phi_k\inv(\Phi_k(p))= T^n\cdot p = \{\Phi_k (p)\} \times (T^n /\Pi_{i\in I} S_i) \subset M_k$ since $\Phi_k(p)\in F_I$. The $S_{k+1}$-action on this set is given as below:
\begin{equation}
[t\eta_{k+1}]\cdot[x_1,\cdots,x_n,\theta_1,\cdots,\theta_n] = [x_1,\cdots,x_n,\theta_1+t\eta^1_{k+1},\cdots, \theta_n+t\eta^n_{k+1}]
\end{equation}
where $t\in\R$ and $[\cdots]$ represent suitable equivalence classes.

Therefore $[t_{k+1}\eta_{k+1}]\cdot[x,\theta] = [x,\theta]$ holds if and only if $[t_{k+1}\eta_{k+1}]$ is an element of the subgroup $\Pi_{i\in I} S_i$. This indicates that there exists an element $[\sum_{i\in I }t_i\eta_i]$ of $\Pi_{i\in I} S_i$ such that $[\sum_{i\in I }t_i\eta_i] + [t_{k+1}\eta_{k+1}]=0$ in $T^n$, that is, $\sum_{i\in I \cup \{k+1\}}t_i\eta_i$ is a point in $2\pi\Z^n$ since $\Delta$ is unimodular. In particular $t_{k+1} \in 2\pi\Z$, then $[t_{k+1}\eta_{k+1}]=0$ in $T^n$. Therefore the $S_{k+1}$-action is free on $\phi_k\inv(\Delta \cap F_{k+1})$. The $S_{k+1}$-action may not be free on $\Phi_k\inv(\R^n\setminus\Delta)$ (but that is locally free, that is, these points have the non-trivial discrete stabilizer), and note that a point with the non-trivial discrete stabilizer yields a orbifold point by the group reduction.  Hence a manifold $M_{k+1}$ that appears in the middle of the construction is an orbifold in generic cases. Nevertheless $M_N$ has no singular point finaly since $\Delta$ is unimodular and so all orbifold points are cut off by the end of the last cutting, therefore that is a manifold.

\end{remark}

\begin{remark}
If $\Delta$ is not polyhedral, i.e. $\Delta = U \cap \{x \in (\R^n)^*\;|\;\langle x,\eta_i\rangle \leqq \kappa_i, i=1,\cdots,N\}$, where $U$ is open set of $(\R^n)^*$, then take $U\times T^n$ instead of $T^*T^n$ in Step $0$ and construct $M_\Delta$ as above.
\end{remark}

\noindent
We can reformulate $N$-times reduction as above to reduction at a time by using the reduction of product groups below:

\begin{theorem}[cutting construction]\label{main}
Let $\Delta := U \cap \{x \in (\R^n)^*\;|\;\langle x,\eta_i\rangle \leqq \kappa_i, i=1,\cdots,N\}$ be a unimodular set, where $U$ is open subset of $(\R^n)^*$. Give the standard coordinate $(x,\theta) = (x_1,\cdots, x_n,$ $\theta_1,\cdots,\theta_n)$ on $U\times T^n$. We consider a Hamiltonian $T^n \times T^N$-space $(U\times T^n \times \C^N ,\sum_{i} dx_i \wedge d\theta_i + \sqrt{-1}\sum_{i=1} ^N{dz_i\wedge d\overline{z}_i},T^n \times T^N,\Psi)$ as follows:\\
$T^n \times T^N$-action is given by
\begin{equation}
(s_1,\cdots, s_n,t_1,\cdots,t_N)\cdot(x,\theta,z_1,\cdots,z_N)=(x,\theta+\sum_{i=1}^n{s_ie_i}+\sum_{i=1} ^N{t_i\eta_i},e^{\i t_1} z_1,\cdots, e^{\i t_N} z_N),
\end{equation}
where $e_i$ is the standard basis of $\R^n$.
Moreover a moment map of this action is described as
\begin{equation}
\Psi:U\times T^n \times \C^N \to (\R^n\oplus\R^N)^*, (x,\theta,z_1,\cdots,z_N) \mapsto x \oplus (\la x,\eta_1\ra+\|z_1\|^2-\kappa_1,\cdots,\la x,\eta_N \ra+\|z_N\|^2-\kappa_N).
\end{equation}
We denote $(\R^N)^*$-part of $\Psi$ by $\mu$. If we take the quotient $M_\Delta := \mu\inv(0)/T^N$, then there are the induced symplectic form $\omega_\Delta$, the induced $T^n$-action and the induced moment map $\Phi_\Delta$ on $M_\Delta$. Moreover $\im \Phi_\Delta = \Delta$ holds.
\end{theorem}

\begin{remark}
This construction is an explicit version of collapse (\cite{KL}) construction for unimodular sets.
\end{remark}

\begin{remark}\label{abcdcba}
By the property of the symplectic cutting (cf.Remark \ref{ab}), we have the action-angle coordinates on $M_0:=\Phi_\Delta \inv(\mathring{\Delta})$ in the straightforward manner.
Let $(T^*T^n,\sum_{i} dx_i \wedge d\theta_i,T^n,\Phi)$ and $(M_\Delta,\omega_\Delta,T^n,\Phi_\Delta)$ be as above. Note that $M_0 = \{(x,\theta,z_1,\cdots,z_N)\in T^*T^n\times\C^N\;|\;\|z_i\|^2=\kappa_i-\langle x,\eta_i\rangle >0, i=1,\cdots,N\}/T^N$ and now consider the map
\begin{equation}
\sigma:\Phi\inv(\mathring{\Delta}) \to M_0,\; (x,\theta)\mapsto[x,\theta,\sqrt{\kappa_1-\langle x,\eta_1\rangle},\cdots,\sqrt{\kappa_N-\langle x,\eta_N\rangle}]
\end{equation}
where $[\cdot,\cdot]$ is the equivalence class of the $T^N$-quotient.\\
One can easily see that $\sigma$ is an equivariant symplectomorphism. Moreover the canonical coordinates $(x,\theta)$ of $T^*T^n$ is the action-angle coordinates. Therefore we can obtain the action-angle coordinates on $M_0$ by pulling back the canonical coordinates $(x,\theta)$ of $T^*T^n$ by $\sigma\inv$. Hence in what follows we take them as the canonical coordinates. We can identified $M_0$ with $\Phi\inv(\mathring{\Delta})$ as a symplectic manifold and we can also consider that $M_0$ is an open symplectic submanifold of $T^*T^n$ in that sense.
\end{remark}

\section{The diffeomorphism type and the fundamental group of $M_\Delta$}
This section deals with diffeomorphism types and homotopy groups of $M_\Delta$ which corresponds to weakly convex unimodular sets.
\begin{theorem}
Let X be a $2n$-dimensional toric variety with fan $(\Sigma,\Z^n)$. Then there exists a $2(n-k)$-dimensional toric variety $B$ such that $X= (\C^*)^k\times B$ if and only if there exists a $(n-k)$-dimensional sub lattice $N$ of $\Z^n$ such that $\Sigma \subset \R$-$\rm{span}(N)$.
\end{theorem}
\begin{proof}
See the book of Fulton$($\cite{F}$)$ p.22.
\end{proof}

We can translate the above theorem in the complex situation to the symplectic situation as a below:

\begin{theorem}
Let $\Delta =  \{x \in (\R^n)^*\;|\;\langle x,\eta_i\rangle \leqq \kappa_i, i=1,\cdots,N\}$ be a weakly convex unimodular set and suppose that $\rm{codim} (\R$-$\rm{span}\{\eta_1,\cdots,\eta_N\} )= k$.
Then we can choose $\{i_1,\cdots,i_{n-k}\} \subset \{1,\cdots,n\}$ such that $\eta_i^\prime := {}^t (\eta_i^{i_1},\cdots,\eta_i^{i_{n-k}})\in\R^{n-k}$ for each $i=1,\cdots,N$ and $\Delta^\prime := \{x \in (\R^{n-k})^*\;|\;\langle x,\eta_i^\prime\rangle \leqq \kappa_i, i=1,\cdots,N\}$ is strongly convex. Moreover $M_\Delta$ diffeomorphic to $(\C^*)^k\times M_{\Delta^\prime}$.
\end{theorem}

\begin{proof}
Set $B = M_{\Delta^\prime}$ and take sublatice $\Z$-$\rm{span}\{\eta_1,\cdots,\eta_N \}$ as $N$  in above theorem then it is clear by the sufficient part of that.
\end{proof}

\begin{corollary}
Let $\Delta =  \{x \in (\R^n)^*\;|\;\langle x,\eta_i\rangle \leqq \kappa_i, i=1,\cdots,N\}$ be a weakly convex unimodular  set with $\rm{codim} (\R$-$\rm{span}\{\eta_1,\cdots,\eta_N\} )= k$.
Then the following holds:
\[
\pi_{m}(M_\Delta)=
\begin{cases}
\pi_m(M_{\Delta^\prime}) & (m\geqq2)\\
\Z^k\times\pi_m(M_{\Delta^\prime}) & (m=1)\\
\{0\} & (m=0)
\end{cases}
\]
Hence if we want to know the homotopy groups of $M_\Delta$ then we have only to consider the homotopy groups of its strongly convex part $M_{\Delta^\prime}$. For example, the $1$-dimensional and $2$-dimensional homotopy groups of the strongly convex good cone were computed in E.Lerman \cite{LL}.
\end{corollary}

\section{application to contact toric manifolds}

Here we discuss the case when a unimodular set $\Delta$ is a polyhedral cone ( or equivalently, good cone ). In this case, $\Delta$ corresponds to a symplectic cone and it include the contact toric manifold $M_\Delta$ as a sub-manifold which is a pre-image of the unit sphere with the moment map image. Note that a symplectic cone is a symplectic manifold $(S,\omega)$ with a free proper action $\{\rho_t\}_{t\in\R}$ of the real line such that $\rho_t^* \omega=e^t\omega$ for all $t\in\R$. In this section, we denote a contact toric manifold which corresponds to a good cone by $M_\Delta$ and its symplectization by $\overline{M}_\Delta$.

According to the Appendix, there exists the canonical K\"{a}hler structure on the symplectization of a contact toric manifold $M_\Delta$, then it is reasonable to ask that whether there exists a toric Sasakian structure in the isomorphism class of a contact toric manifold $M_\Delta$.
\begin{remark}
It is clear that the above canonical K\"{a}hler metric is not cone metric (c.f. \cite{MSY}, p.44), hence it does not give $M_\Delta$ a Sasakian structure, but it is left that a possibility of $\overline{M_\Delta}$ have an another cone K\"{a}hler structure with respect to which $M_\Delta$ is Sasakian.
\end{remark}

To solve this problem we prove Theorem \ref{main2} in below. ( Then we can see that there exists a toric K-contact structure in the isomorphism class of a contact toric manifold if and only if its moment cone is strongly convex.)

Generally speaking, if a contact manifold is of toric Sasakian type then it is also of toric K-contact type. Hence the above theorem says $M_\Delta$ which corresponds to weakly convex good cone is not of toric Sasakian type.

\subsection{Preliminaries}
We recall some basic notions of the contact geometry. This subsection and the next one are based on \cite{BG,Bl,Bo}.

Let $M$ be an orientable $(2n-1)$-dimensional manifold and suppose $\xi$ is co-dimension one sub bundle of $TM$.
Then the pair $(M,\xi)$ is a contact manifold if and only if $\xi$ is a maximally non-integrable distribution, that is, $\L_X(\Gamma(\xi))\nsubseteq\Gamma(\xi)$ holds for any non-zero section $X\in\Gamma(\xi)$, where $\Gamma(\xi)$ is the space of sections of $\xi$. we call $\xi$ the contact distribution.

Let $\xi^\circ$ be the annihilator line bundle of $\xi$, and assume we can take a nowhere vanishing section $\alpha$ of $\xi^\circ\subset T^*M$, i.e., $\xi$ is co-orientable. Then 1-form $\alpha$ have the following properties:
\begin{enumerate}
\item $\alpha\wedge(d\alpha)^{n-1}\neq0$, $\xi=\ker\alpha$;
\item there exists a unique vector field $R_\alpha$ such that $\alpha(R_\alpha)=1$ and $\iota_{R_\alpha} d\alpha = 0$;
\item $R_\alpha$ generates a trivial line bundle $L_\alpha$ and the characteristic foliation $\F_\alpha$;
\item $L_\alpha$ provides a splitting of tangent bundle $TM=\xi\oplus L_\alpha$.
\end{enumerate}

We call $\alpha$ satisfying $(1)$ a {\bf contact form} representing a contact structure $\xi$ and a vector field $R_\alpha$ on $M$ which satisfies $(2)$ a {\bf Reeb vector field} of $\alpha$. Indeed, for any nowhere vanishing function $f$, $\alpha^\prime:=f\cdot\alpha$ is also a contact form of the same distribution $\xi$. On the other hand, if $\alpha$ and $\alpha^\prime$ are contact forms of the same distribution then there exists nowhere vanishing function $f$ such that $\alpha^\prime:=f\cdot\alpha$. Additionally, we determine the co-orientation of $M$ by choosing the positive part of $\xi^\circ \setminus\{0$-$section\}$ and we denote this $\xi^\circ _+$. It is clear that $\alpha,\alpha^\prime\in\Gamma(\xi^\circ_+)\iff {}^\exists f>0$ s.t. $\alpha^\prime=f\cdot\alpha$.

\subsection{The group of contactomorphisms}
In this section, we prepare some lemmata about the group of contactomorphisms for the proof of Theorem \ref{main2}
Let $\Diff(M)$ denote the group of diffeomorphisms of $M$. We endow $\Diff(M)$ with the compact-open $C^\infty$-topology in which case it becomes a Fr\'{e}chet Lie group. (See \cite{M,Ba,O,KM} for more details.)
\\

Let $(M,\xi)$ be a closed connected contact manifold. Then we define the group $\Con(M,\xi)$ of all {\bf contact diffeomorphisms} or {\bf contactomorphisms} by
\begin{equation*}
\Con(M,\xi)=\{\phi\in\Diff\;|\;\phi_*\xi\subset \xi\}.
\end{equation*}
If $\alpha$ is a contact form, then it is easy to see that
\begin{equation*}
\Con(M,\xi)=\{\phi\in\Diff(M)\;|\;\phi^*\alpha=f^\alpha_\phi\alpha, \;f^\alpha_\phi\in C^\infty(M)^*\},
\end{equation*}
where $C^\infty(M)^*$ denotes the subset of nowhere vanishing functions in $C^\infty(M)$.\\

\noindent
We also consider the subgroup $\Con(M,\alpha)$ of {\bf strict contact transformations}:
\begin{equation*}
\Con(M,\xi)^+=\{\phi\in\Diff(M)\;|\;\phi^*\alpha=\alpha\}.
\end{equation*}
This is the group of all contactomorphisms whith preserve a contact form $\alpha$.\\

We can consider the Lie algebra of those groups as infinitesimal transformations:
\begin{eqnarray*}
\con(M,\xi)&:=&\{X\in\Gamma(TM)\;|\;\L_X\alpha=f_X\alpha, \;f_X\in C^\infty(M)\};\\
\con(M,\alpha)&:=&\{X\in\Gamma(TM)\;|\;\L_X\alpha=0\}.
\end{eqnarray*}

A choice of contact forms $\alpha\in\Gamma(\xi^\circ_+)$ gives a well known Lie algebra isomorphism between $\con(M,\xi)$ and $C^\infty(M)$ given explicitly by $X\mapsto\alpha(X)$, where the Lie algebra structure on $C^\infty(M)$ given by the Poisson-Jacobi bracket. Then one can easily see that the sub-algebra $\con(M,\alpha)$ is isomorphic to the sub-algebra of $F$-invariant functions $C^\infty(M)^F$, where $F$ is a flow of the Reeb vector field $R_\alpha$ of $\alpha$. In particular, $R_\alpha$ is included in $\con(M,\alpha)$.\\

\noindent

\noindent
In fact, the following holds:

\begin{lemma}\label{caa}
The centralizer of a Reeb vector field $R_\alpha$ in $\con(M,D)$ is $\con(M,\alpha)$.
\end{lemma}
\begin{proof}
See \cite{Bo}.
\end{proof}


\subsection{Contact toric manifolds}
In this subsection, we introduce some basic notions of the contact toric geometry.

\begin{definition}
Let $(M,\xi)$ be a contact manifold, let $T$ be a torus, and suppose $T$ acts on $M$. Then a triplet $(M,\xi,T)$ is a contact toric manifold if and only if the torus $T$ is embedded in $\Con(M,\xi)$ and $\dim T$ is equal to $\frac{1}{2}(\dim M +1)$.
Now fixing a contact form $\alpha$. Then the $\alpha$-moment map $\Phi_\alpha$ is defined by $\Phi_\alpha:M\to \ft^*,$ $\langle\Phi_\alpha(x),X\rangle = \alpha(X)(x)$.
\end{definition}

We note that the $\alpha$-moment map depends on a particular choice of a contact form and not just on the contact structure. Hence we take more \textit{canonical} or \textit{universal} moment map as follows:
\begin{itemize}
\item Note that we can take a $T$-invariant contact form $\alpha\in\Gamma(\xi^\circ_+)$ by averaging $\alpha$ over the torus $T$ and hence we can assume $T\subset\Con(M,\alpha)$. Since the image of arbitrary $\alpha^\prime$-moment map do not in fact contain the origin, we can take a normalized contact form $\alpha := \frac{\alpha^\prime}{\|\Phi_{\alpha^\prime}\|}$ and then $\|\Phi_\alpha\|\equiv1$. We call this the {\bf canonical moment map}.
\item In order to treat the contact $1$-forms on an equal footing, we consider $\xi^\circ_+$ as a symplectization of $M$ and take a moment map $\Psi:\xi^\circ_+\to\ft^*$ given by $\langle\Psi(x,\alpha_x),X\rangle=\alpha(X)(x)$. We call this the {\bf universal moment map}. The universal moment map is a moment map which is defined in symplectic toric geometry's context of a symplectization of a contact toric manifold $(M,\xi,T)$.
 \end{itemize}

\begin{definition}
Let $(M,\xi,T,\Psi)$ be a contact toric manifold as above. We define the moment cone $C(\Psi)$ to be the set
\begin{equation*}
C(\Psi):={\rm Im}\Psi\cup\{0\}.
\end{equation*}
\end{definition}

Note that symplectic cone is non-compact toric symplectic manifold. Hence ${\rm Im} \Psi$ is a unimodular set. In particular we can denote $C(\Psi)=\{x \in \ft^*\;|\;\langle x,\eta_i\rangle \geqq 0, i=1,\cdots,N\}$ for some $\eta_i \in \mathfrak t$, $i =1,\cdots, N$.

\begin{definition}\label{good}
A polyhedral cone $C$ is {\bf good} if and only if $C\setminus\{0\}$ is a unimodular set.
\end{definition}

There is well known Lerman's classification theorem of c.c.c.t. (compact connected contact toric) manifolds \cite{L2}. From this theorem, there is a one-to-one correspondence between c.c.c.t. manifolds with a non-free toric action and good cones. 

A c.c.c.t. manifold $(M,\alpha,T,\Phi_\alpha)$ with the canonical contact form is embedded in the symplectic cone $(\xi^\circ_+,d(e^t\alpha))$ as a pre-image of the intersection of a moment cone $C(\Psi)$ and the unit sphere. In fact the symplectization commutes with the symplectic/contact cutting (see \cite{L1}). Hence we can construct every c.c.c.t. manifold with non-free toric action by the cutting construction similarly to the symplectic case. This proves Theorem \ref{contact}.

\begin{remark}
We can take a contact version of action-angle coordinates on some open dense subset:\\
Let $\Delta$ be a good cone and let $\Psi$ be the moment map as in Theorem \ref{contact}. Let $\Phi$ be a $(\R^n)^*$-part of $\Psi$, that is, a moment map of $(S^*(T^n),\sum{x_id\theta}, T^n)$ and let $\Phi_\Delta$ is an induced moment map as in Theorem \ref{contact}. Then, on a open dense subset $\Phi_\Delta\inv(\mathring{\Delta}\cap S^{n-1})$, since a contact structures on $\Phi_\Delta\inv(\mathring{\Delta}\cap S^{n-1})$ and on $\Phi\inv(\mathring{\Delta}\cap S^{n-1})$ are contactomorphic, we can denote a contact form $\alpha_\Delta=\sum_{i=1}^n x_i d\theta_i$ and a Reeb vector field ${R_{\alpha_\Delta}}=\sum_{i=1}^n{ x_i\frac{\partial}{\partial \theta_i}}$.
\end{remark}

\begin{remark}
One can easily see the compactness and connectedness of $M_\Delta$ in the above theorem by observing the pre-image of the moment map $\mu$ and $\Phi_\Delta$.
\end{remark}

\begin{definition}
Let $(M,\xi)$ be a contact manifold. We say that a torus $T\subset\Con(M,\xi)$ is of {\bf Reeb type} if there is a contact $1$-form $\alpha$ such that its Reeb vector field lies in the Lie algebra $\ft$ of $T$.
\end{definition}

\subsection{Contact toric metric structure}
In this subsection, we introduce the definition of ``of toric K-contact/Sasakian type". Let $(M,\xi)$ be a contact manifold. Let $(\alpha,\Phi,g)$ be a contact metric structure on $M$ with a contact form $\alpha$, an almost CR structure $\Phi$, and a metric $g$. We call the quadruplet $(M,\alpha,\Phi,g)$ a contact metric manifold. For more details of the definition of the contact metric structure, see \cite{Bl, Bo}.\\

\noindent
A contact metric manifold $(M,\alpha,\Phi,g)$ is said to be a K-contact manifold if the Reeb vector field $R_\alpha$ is a Killing vector field. Moreover,
a K-contact manifold $(M, \alpha, \Phi, g)$ said to be a Sasakian manifold if its symplectic cone $(\R\times M,d(e^t\alpha),J_c)$ is a K\"{a}hler cone (that is, $J_c$ is integrable, c.f. \cite{BG}), where $t$ is a coordinate of the radial direction of the cone and $J_c$ is an almost complex structure which is given by the following:
\begin{equation*}
J_c(X\oplus f \frac{\partial}{\partial t}):= (\Phi(X)-f\xi)\oplus\alpha(X)\frac{\partial}{\partial t},\; (X\in\Gamma(TM),f\in C^\infty(D^\circ_+)).
\end{equation*}

\begin{definition}
A contact manifold $(M,\xi)$ is of {\bf K-contact type} if it admits some K-contact structure, and a contact manifold $(M,\xi)$ is of {\bf Sasakian type} if it admit some Sasakian structure as well. Moreover an contact toric manifold $(M,\xi,T)$ is of {\bf toric K-contact type} if it admits some K-contact structure $(\alpha,\Phi,g)$ which consist of $T$-invariant form, almost CR-structure, and metric. We call the quintuplet $(M,\alpha, T, \Phi, g)$ toric K-contact manifold. Similarly, an contact toric manifold $(M,\xi,T)$ is of {\bf toric Sasakian type} if it admits some Sasakian structure $(\alpha,\Phi,g)$ which consist of $T$-invariant form, almost CR-structure, and metric. We call the quintuplet $(M,\alpha, T, \Phi, g)$ toric Sasakian manifold.
\end{definition}
%

\subsection{Strongly convex cones and manifolds of toric K-contact type}
First, we give some lemmata for proving Theorem \ref{main2} in below. 
\begin{lemma}
Let $(M,\xi)$ be a contact manifold and take a torus $T\subset\Con(M,\xi)$. Then the action of the torus $T$ is of Reeb type if and only if there exist an $X \in \mathfrak t$ and a contact form $\alpha \in \Gamma(\xi^o_+)$ such that  $\langle \Phi_\alpha, X\rangle > 0$, where $\Phi_\alpha$ is the $\alpha$-moment map.
\end{lemma}
\begin{proof}
If the torus action is of Reeb type, then we can take a contact form $\alpha^\prime\in\Gamma(\xi^\circ_+)$ such that its Reeb vector field $R_{\alpha^\prime}$ is contained $\ft$. Since both $\alpha$ and $\alpha^\prime$ are elements of $\Gamma(\xi^\circ_+)$, there is a positive function $f$ such that $\alpha=f\cdot\alpha^\prime$. Therefore $\langle \phi_\alpha,\xi^\prime\rangle=f\cdot\alpha^\prime(\xi^\prime)=f>0$. Hence only if part holds. If there exist $X\in\ft$ and $\alpha$ such that $\langle\Phi_\alpha,X\rangle>0$. Without loss of generality, we can assume $\alpha$ is $T$-invariant by averaging $\alpha$ on $T$. Now we set $\alpha^\prime:=\alpha/\alpha(X)\in\Gamma(\xi^\circ_+)$, then one can easily see that $\alpha^\prime(X)=1$ and $\iota_Xd\alpha^\prime=0$ holds. This indicates $R_{\alpha^\prime} = X$. Therefore torus $T$ is of Reeb type.
\end{proof}

\begin{corollary}\label{bb}
Let $(M,\xi)$ and $T$ be as above. Let $\Psi$ be a universal moment map of $(M,\xi,T)$. Then the action of the torus $T$ on $M$ is of Reeb type if and only if there exists $X\in\ft$ such that $\langle\Psi,X\rangle>0$.
\end{corollary}

\begin{lemma}
Let $(M,\xi,T,\Psi)$ be a c.c.c.t. manifold. Then the action of the torus $T$ on $M$ is of Reeb type if and only if its moment cone $C(\Psi)$ is a strongly convex cone.
\end{lemma}
\begin{proof}
If $C(\Psi)$ is weakly convex, then there exists $X\in\ft$ and $x\in{\rm Im}\Psi\setminus\{0\}$ such that $\langle x,X\rangle=0$. Hence we can not take $X\in\ft$ such that $\langle\Psi,X\rangle>0$. Therefore torus $T$ is of non-Reeb type. Suppose $C(\Psi)$ is strongly convex. Note that we can denote $C(\Psi)=\{x \in \ft^*|\langle x,\eta_i\rangle \geqq 0, i=1,\cdots,N\}$. Now take $X\in\sum_i^N a_i\eta_i$ for positive numbers $a_1,\cdots,a_N$. Then since $C(\Psi)$ is strongly convex, $\langle x,X\rangle=\sum a_i\langle x,\eta_i\rangle>0$ for any $x\in{\rm Im}\Psi$. The result now follows from Corollary \ref{bb}.
\end{proof}

\begin{corollary}\label{ee}
Let $(M,\xi,T,\Psi)$ be as above. There is a sub-torus $K\subset T$ of Reeb type if and only if $C(\Psi)$ is a strongly convex.
\end{corollary}
\begin{proof}
The moment map of $K$-action is given by $i^*\circ\Psi$ where $i:K\to T$ is inclusion. Since $i^*$ is a projection from $\ft$ to $\fk$, we get the result.
\end{proof}

\begin{lemma}\label{dd}
Let $(M,\xi)$ be a $(2n-1)$-dimensional contact manifold. Then the dimension of a maximal torus in $\Con(M,\xi)$ is at most $n$ where maximality is given by inclusion property.
\end{lemma}
\begin{proof}
Consider the lift of an action of $T\subset\Con(M,\xi)$ on $M$ to the $T$-action on the symplectization $(\xi^\circ_+,d(e^t\alpha))$ where $\alpha$ is some contact form and $t$ is the radial coordinate. Since the symplectic form $d(e^t\alpha)$ is exact, $T$ is a subset of ${\rm Ham}(\xi^\circ_+,d(e^t\alpha))$. The dimension of the torus acting in an effective Hamiltonian way on symplectic manifold $X^{2n}$ is at most $n$ (see \cite{C}, Theorem 27.3). Therefore $\dim T \leqq n$ holds.
\end{proof}

\begin{proposition}\label{zxc}
Let $(M,\alpha,T,\Phi,g)$ be a compact connected toric K-contact manifold with contact form $\alpha$, a CR structure $\Phi$ and a metric $g$. Then $C(\Psi)$ is a strongly convex.
\end{proposition}
\begin{proof}
Note that a contact form $\alpha$ and a metric $g$ are $T$-invariant. A contact metric manifold $(M,\alpha,\Phi,g)$ is K-contact if and only if $\L_\xi g = 0$, that is, a Reeb flow $F$ is contained by ${\rm Isom}_0(M,g)$, where ${\rm Isom}_0(M,g)$ is the identity component of the isometry group. The following result is well known: if $M$ is compact Riemanian manifold then the isometry group ${\rm Isom}_0(M,g)$ is a finite dimensional compact Lie group. Therefore $F$ is a subgroup of the compact finite dimensional group $G:={\rm Isom}_0(M,g)\cap\Con_0(M,\alpha)$. Moreover $T$ is included in $G$. Thus $F$ is contained in the centralizer of $T$ since $F$ is in the center of $G$ as follows from Lemma \ref{caa}. On the other hand, Lemma \ref{dd} yields that $T$ is maximal and in fact that the centralizer of a maximal torus is coincide to itself. Therefore $F$ is contained in torus $T$, that is, ${\rm Lie}(F) \subset \ft$. Now consider the closure $\overline{F}:=$ the closure of $F$, then this is a sub-torus in $T$ of Reeb type. The result now follows from Corollary \ref{ee}.
\end{proof}
\noindent
\textit{Proof of Theorem \ref{main2}} By Proposition \ref{zxc} we have only to prove its converse. If the moment cone of $(M,\xi,T,\Psi)$ is strongly convex, then we get a toric Sasakian structure on $(M,\xi,T,\Psi)$ by Delzant construction as \cite{L2}. Hence $(M,\xi)$ is of toric Sasakian type, in particular of toric K-contact type. This completes the proof of Theorem \ref{main2}.\\
As a result, strongly convex good cones correspond to of toric K-contact type, in particular of toric Sasakian type and weakly convex cones correspond to of toric non-Sasakian type contact manifold.

\begin{remark}
In toric cases, it is true that toric K-contact type implies toric Sasakian type, since Theorem \ref{main2} holds, and we can construct a toric Sasakian structure which is induced by Delzant construction, but it is not true in generic cases.
\end{remark}




Moreover the symplectic cones that correspond to a weakly convex cone do not have a K\"{a}hler cone structure, but they have the canonical K\"{a}hler structure that is determined by the cutting construction and hence c.c.c.t.manifolds of toric non-Sasakian type have the canonical almost contact metric structure as follows:

Let $(M,\xi,T,\Psi)$ be a c.c.c.t. manifold with the canonical contact form $\alpha$ and let $(J,d(e^t\alpha),h)$ be the canonical K\"{a}hler structure  on the symplectization $(\R\times M,d(e^t\alpha),T,\Psi)$ which is determined by the cutting construction. Now $M$ is embedded in its symplectic cone as the pre-image of the intersection of the unit sphere and a moment cone under the moment map $\Psi$. Then we get an almost contact metric structure $(\Phi,\xi,\alpha,g)$ which satisfies:

\begin{equation*}
J\iota_*X=\iota_*\Phi X + \alpha(X)\xi,\;g:=\iota^*h,
\end{equation*}
where $\iota:M\to D^\circ_+$ is the inclusion map and $X\in\Gamma(TM)$.

\begin{remark}
Generally speaking, every $C^\infty$ orientable hypersurface of an almost complex manifold has an almost contact structure and if its ambient space is an almost Hermitian manifold then it has an almost contact metric structure $($see \cite{T} $)$. 
\end{remark}


\setcounter{section}{5}
\setcounter{theorem}{0}

\section*{Appendix}
In this section, we give a expression of the canonical K\"{a}hler structure on cutting constructed symplectic toric manifolds as follows:
\begin{theorem}\label{spsp}
Let $\Delta = \{x \in (\R^n)^*\;|\;\langle x,\eta_i\rangle \leqq \kappa_i, i=1,\cdots,N\}$ be a unimodular set, and suppose $(M_\Delta,\omega_\Delta,T^n,\Phi_\Delta)$ is a symplectic toric manifold which is constructed by the cutting construction. Then there is a canonical K\"{a}hler structure $(\omega_\Delta,J_\Delta,g_\Delta)$ on $M_\Delta$ which is given by the symplectic potential;
\begin{equation*}
\Sp(x)=\frac{1}{2}\|x\|^2+\frac{1}{2}\sum{l_i(x)\log l_i(x) - \frac{1}{2}l_\infty(x)}.
\end{equation*}
\end{theorem}

First, the cotangent bundle of $n$-dimensional torus $T^*T^n$ have the natural K\"{a}hler structure. More explicitly, if we give coordinates $(T^*T^n,x_1,\cdots ,x_n,\theta _1,\cdots ,\theta _n )$, then we can take a symplectic form $\omega_{ST} :=\sum_idx_i\wedge d\theta_i$ and  a compatible almost complex structure $J_{ST}$ as follows:

\begin{equation}
J_{ST}:T^* M \rightarrow T^* M,\; \frac{\partial}{\partial x_i} \mapsto \frac{\partial}{\partial \theta _i}, \frac{\partial}{\partial \theta _i} \mapsto -\frac{\partial}{\partial x_i},
\end{equation}
where $M$ is $T^* T^n$.

Hence $M_\Delta$ has the canonical K\"{a}hler structure that is induced by the standard K\"{a}hler structure of $T^*T^n$ and the cutting construction. Now we compute the induced K\"{a}hler structure of $M_\Delta$ by applying the following proposition.

\begin{proposition}\label{mainp}
Let $\Delta = \{x\in(\R^n)^*\;|\;\langle x,\eta_i \rangle \leqq \kappa, i = 1,\cdots,N \}$ be a unimodular set and let $(M_\Delta,\omega_\Delta,T^n,\Phi_\Delta)$ be the Hamiltonian space as in Theorem \ref{main}. Then there exists a biholomorphic map $g$ between $(M_\Delta \supset )M_0 =\Phi_\Delta \inv(\mathring{\Delta})$ with the induced complex structure and $T^*T^n$ with the standard complex structure which can be expressed in terms of the action angle coordinates, and such a biholomorphism $g$ is given by the following formula

\begin{equation}
g:M_0 \rightarrow T^*T^n,\; (x,\theta)\mapsto(x-\frac{1}{2}\sum_{i=1}^{N}[\log(\kappa _i - \langle x,\eta _i \rangle)]\eta _i,\theta)
\end{equation}
where $x=(x_1,\cdots,x_n)$, $\theta = (\theta_1,\cdots,\theta_n)$ and since the symplectic form induces the natural isomorphism $\lie{t}^n \cong \R^n \cong {\lie{t}^n}^*$ we may consider $\eta_i$ and $x$ in the same space.\end{proposition}

\begin{proof}
For simplicity, we set $(M,\omega)=(T^*T^n,\sum_{i} dx_i \wedge d\theta_i)$. We first note the following.\\
(1)  Let $(x,\theta)$ be a point of $M_0$ and suppose $t=(t_1,\cdots,t_N) \in \R^N$ satisfies 
\begin{equation*}
(\kappa _1,\cdots,\kappa_N)-(e^{2t_1},\cdots,e^{2t_N}) = \phi(x,\theta) = (\langle x,\eta _1\rangle,\cdots,\langle x,\eta _N \rangle),
\end{equation*}
then $\exp(\sum_{i=1}^{N}t_i \nabla\phi_i)((y,\xi))=(x,\theta)$ if and only if $(y,\xi)=g(x,\theta)$ where $\phi _i$ is the $i$-th component of $\phi$, that is $\langle x,\eta_i \rangle$.

We prove $g$ is a biholomorphism from $M_0$ to its image. To see this, we consider the following biholomorphic map

\begin{equation}
f:M \times (\C^*)^N \rightarrow M \times(\C^*)^N,\; (p,z)\mapsto(z \cdot p, z)
\end{equation}
where $z=(z_1,\cdots,z_N)=(e^{r_1+\sqrt{-1}s_1},\cdots,e^{r_N+\sqrt{-1}s_N})$, $p=(x,\theta)$, and $z\cdot p :=(x+\Sigma_i r_i\eta_i,\theta + \Sigma_i s_i \eta_i)$.

Give a diagonal action of $(\C^*)^N$ on $M \times (\C^*)^N$. Now the pullback by $f$ of the K\"{a}hler form 
$\omega + \sqrt{-1} \sum_{i=1} ^{N} dz_i \wedge d\overline{z_i}$ is invariant under the diagonal action of that $T ^N( \subset (\C^*)^N)$. Moreover the diagonal $T^N$-action is the Hamiltonian action with a moment map $\overline{\mu}(p,z) = \phi(z \cdot p) + (\| z_1 \|^2,\cdots,\| z_N \|^2) - (\kappa_1,\cdots,\kappa_N)$. In particular, $f$ maps the level set $\overline{\mu}\inv(0)$ to $\mu\inv(0)$, where $\mu$ is the same map as in Theorem \ref{main}
and induce the K\"{a}hler isometry

\begin{equation}
h:\overline{\mu}\inv(0)/T^N \rightarrow \mu\inv(0)/T^N.
\end{equation}

To describe this isometry more explicitly, take the following subsets
\begin{equation}\label{c}
\{[(x,\theta),(e^{t_1},\cdots,e^{t_N})]\;|\;\langle x + \sum_{k=1}^{N} t_k \eta_k ,\eta_i \rangle + e^{2t_i} = \kappa _i, i=1,\cdots,N\} (\subset \overline{\mu}\inv(0)/T^N),
\end{equation}

\begin{equation}\label{d}
M_0 =\{[(x,\theta),(e^{t_1},\cdots,e^{t_N})]\;|\;\langle x,\eta_i \rangle + e^{2t_i} = \kappa _i, i=1,\cdots,N\}(\subset \mu\inv(0)/T^N)
\end{equation}
where $[\cdot,\cdot]$ represents suitable equivalence classes.

Suppose $h$ maps a point $[(y,\xi),(e^{s_1},\cdots,e^{s_N})]$ in (\ref{c}) to a point $[(x,\theta),(e^{t_1},\cdots,e^{t_N})]$ in (\ref{d}), then
\begin{equation*}
s_i = t_i,\; \theta = \xi,\; x = y + \sum_{k=1}^{N}t_k\eta_k,\; and \;e^{2t_i} = \kappa _i - \langle x,\eta \rangle \; (\mathrm{for}\; i=1,\cdots,N).
\end{equation*}
Therefore $(y,\xi) = g(x,\theta)$ holds since (1) and $\exp(\sum_{i=1}^{N}t_i \nabla\phi_i)((y,\xi))=(x,\theta)$ iff $x = y + \sum_{k=1}^{N}t_k\eta_k$ and $\theta = \xi$. That indicates $h = g\inv$. Therefore $g:M_0 \rightarrow {\rm Im} (g)$ is a biholomorphim between (\ref{c}) and (\ref{d}) with the induced complex structures. Note that open set (\ref{c}) have the same complex structure as $M$ by GIT-quotient. 

Next we show $g$ is surjective. To see this, we take the following map:

\begin{equation}
\tilde{g}:\mathring{\Delta} \rightarrow (\R^n)^*, x\mapsto x-\frac{1}{2}\sum_{i=1}^{N}\log(\kappa_i -\langle x,\eta_i \rangle)\eta_i
\end{equation}
where $\mathring{\Delta}$ denotes the interior of $\Delta$.

Then, the following diagram is commutative:
\begin{equation*}
\begin{CD}
M_0@>g>>M\\
@VV{\Phi}V @VV{\Phi}V\\
\mathring{\Delta}@>{\tilde{g}}>>({\R^n)^*}\\
\end{CD}
\end{equation*}
where, we must consider $M_0 = \Phi \inv (\mathring{\Delta}) \subset M$.\\
It is clear that $g$ is surjective iff $\tilde{g}$ is surjective.
Hence the remains of this proof is to prove $\im (\tilde{g}) = (\R^n)^*$. To see this, we remark $\tilde{g}$ is diffeomorphism to its image because $g$ is diffeomorphism to its image. As a result, $\im(\tilde{g})$ is an n-dimensional manifold in $(\R^n)^*$. In particular, $\im (\tilde{g})$  is open. Suppose $\im(\tilde{g})$ have some boundary points, then that is corresponding to the boundary of $\Delta$ since $\im(\tilde{g})$ is homeomorphic to $\Delta$. However definition of $\im(\tilde{g})$ shows that the boundary of $\Delta$ is mapped on to infinity. Therefore the boundary of $\im(\tilde{g})$ is empty, that is, $\tilde{g}$ is surjective.
\end{proof}

\begin{remark}
This proposition is the explicit version of K\"{a}hler cuts in \cite{BGL}.
\end{remark}

\noindent
\textit{Proof of Theorem \ref{spsp}} We compute the K\"{a}hler structure $(\omega_\Delta,J_\Delta,g_\Delta)$ of $M_\Delta$ on $M_0$ by applying Proposition \ref{mainp}. First of all, note that we have already get action-angle coordinates $(x,\theta)$ on $M_0$ by the symplectic cutting construction in Remark \ref{abcdcba} sense. Hence we describe the complex structure $J_\Delta$ and the metric $g_\Delta$ in action-angle coordinates, i.e, we consider a tangent space of $M_0=\R$-$\rm{span}\{\frac{\partial}{\partial x_1},\cdots, \frac{\partial}{\partial x_n},\frac{\partial}{\partial \theta_1},\cdots,\frac{\partial}{\partial \theta_n}\} \cong$ a tangent space of $T^*T^n$. Now we note that $\omega$, $J_{ST}$ and the derivation of the biholomorphism $g$ is represented by the following matrices:\\
we set
\begin{equation*}
\eta_1=
\begin{pmatrix}
\eta_1^1\\
\cdot\\
\cdot\\
\cdot\\
\eta_1^n
\end{pmatrix}
,\cdots,\eta_N=
\begin{pmatrix}
\eta_N^1\\
\cdot\\
\cdot\\
\cdot\\
\eta_N^n
\end{pmatrix}
,l_1(x):= \kappa_1 - \langle x,\eta_1 \rangle,\cdots,l_N(x):= \kappa_N - \langle x,\eta_N \rangle, l_\infty(x):= \sum{l_i(x)},
\end{equation*}
and
\begin{equation*}
G=
\begin{pmatrix}
1+ \frac{1}{2}\sum\frac{(\eta_i^1)^2}{l_i(x)} & \frac{1}{2}\sum_i\frac{\eta_i^1\eta_i^2}{l_i(x)} & \cdots & \frac{1}{2}\sum_i\frac{\eta_i^1\eta_i^n}{\l_i(x)}\\
\cdot\\
\cdot\\
\cdot\\
\frac{1}{2}\sum\frac{\eta_i^1\eta_i^n}{l_i(x)} & \cdots & \frac{1}{2}\sum_i\frac{\eta_i^{n-1}\eta_i^n}{l_i(x)} & 1+ \frac{1}{2}\sum_i\frac{(\eta_i^n)^2}{l_i(x)}
\end{pmatrix}
.
\end{equation*}
Then,
\begin{equation}
\omega=
\begin{pmatrix}
O & I\\
-I & O
\end{pmatrix}
,J_{ST}=
\begin{pmatrix}
O & -I\\
I & O
\end{pmatrix}
,g_*=
\begin{pmatrix}
G & O\\
O & I
\end{pmatrix}
.
\end{equation}

Then, since $\omega_\Delta = \omega$, $J_\Delta = g_*\inv\circ J_{ST}\circ g_*$ and $g_\Delta(\cdot,\cdot)=\omega_\Delta(\cdot,J_\Delta\cdot)$, they are represented by the following matrices:
\begin{equation}\label{f}
\omega_\Delta=
\begin{pmatrix}
O & I\\
-I & O
\end{pmatrix}
,J_\Delta=
\begin{pmatrix}
O & -G\inv\\
G & O
\end{pmatrix}
,g_\Delta=\omega_\Delta(\cdot,J_\Delta \cdot)=
\begin{pmatrix}
G & O\\
O & G\inv
\end{pmatrix}
.
\end{equation}

Now we wish to take a function, denoted by $\mathrm{Sp}$, whose Hessian matrix provides $G$. By a direct calculation one can show that this function $\rm{Sp}$ $:\mathring{\Delta}\to\R$ is given by the following formula:
\begin{equation}
\Sp(x):=\frac{1}{2}\|x\|^2+\frac{1}{2}\sum{l_i(x)\log l_i(x) - \frac{1}{2}l_\infty(x)}.
\end{equation}
We call this function the \textit{symplectic potential} of the canonical K\"{a}hler structure. This completes the proof of Theorem \ref{spsp}.
\begin{remark}
A symplectic potential is introduced by V.Guillemin $($\cite{G}$)$ as the Legendre transformation of a K\"{a}hler potential. In our situation, the function $\rm{Sp}$ is determined by the construction rather than the Legendre transformation, but, nevertheless we can call it the symplectic potential for similarity of between representation (\ref{f}) and usual representation of the K\"{a}hler structure on action-angle coordinates $($c.f.\cite{A}$, p.7)$.
\end{remark}

\begin{remark}
If we drop $l_\infty (x)$ from $\Sp(x)$, it define the same K\"{a}hler structure as $(\ref{f})$, but its derivation is not equal to $\tilde{g}$. Similarly there are other deformations of $\Sp$ whose Hessian is non-degenerate. Specifically if $\Delta$ is polytope, then $\Sp(x) - \frac{1}{2}\|x\|^2$ coincide with the well known Guillemin's symplectic potential. On the other hand, if $\Delta$ is weakly convex, then we can not drop $\frac{1}{2}\|x\|^2$ from $\Sp(x)$ because if we drop it, the Hessian is degenerate, that is, it does not define K\"{a}hler structure.
\end{remark}

\begin{remark}
If a unimodular set $\R^n\supset\Delta$ is the n-simplex then the fixed point of $\tilde{g}$ coincides with the barycenter.
\end{remark}

In one-time cutting cases, we can compute $G\inv$ explicitly:
\begin{equation}
G\inv=\frac{1}{1+\frac{\|\eta\|^2}{\kappa-\langle x, \eta\rangle}}
\begin{pmatrix}
1+ \frac{1}{2}\frac{\sum_{i\neq1}(\eta^i)^2}{\kappa-\langle x, \eta\rangle} & -\frac{1}{2}\frac{\eta^1\eta^2}{\kappa-\langle x, \eta\rangle} & \cdots & -\frac{1}{2}\frac{\eta^1\eta^n}{\kappa-\langle x, \eta\rangle}\\
\cdot\\
\cdot\\
\cdot\\
-\frac{1}{2}\frac{\eta^1\eta^n}{\kappa-\langle x, \eta\rangle} & \cdots & -\frac{1}{2}\frac{\eta^{n-1}\eta^n}{\kappa-\langle x, \eta\rangle} & 1+ \frac{1}{2}\frac{\sum_{i\neq n}(\eta^i)^2}{\kappa-\langle x, \eta\rangle}
\end{pmatrix}
.
\end{equation}



\begin{thebibliography}{n}
\bibitem[A]{A}M. Abreu, K\"{a}hler geometry of toric manifolds in symplectic coordinates. Symplectic and contact topology: interactions and perspectives (Toronto, ON/Montreal, QC, 2001), 1--24, Fields Inst. Commun., 35, Amer. Math. Soc., Providence, RI, 2003.
\bibitem[Ba]{Ba}A. Banyaga, The structure of classical diffeomorphism groups. Mathematics and its Applications, 400. Kluwer Academic Publishers Group, Dordrecht, 1997. xii+197 pp. ISBN: 0-7923-4475-8
\bibitem[BG]{BG}C.P. Boyer,  K. Galicki, Sasakian geometry. Oxford Mathematical Monographs. Oxford University Press, Oxford, 2008. xii+613 pp. ISBN: 978-0-19-856495-9
\bibitem[Bl]{Bl}D. Blair, Riemannian geometry of contact and symplectic manifolds. Second edition. Progress in Mathematics, 203. Birkh\"{a}user Boston, Inc., Boston, MA, 2010. xvi+343 pp. ISBN: 978-0-8176-4958-6
\bibitem[Bo]{Bo}C.P. Boyer, Maximal Tori in Contactomorphism Groups. arXiv:1003.1903v2
\bibitem[BGL]{BGL}D. Burns, V. Guillemin and E. Lerman, Kaehler cuts. arXiv:math/0212062v1
\bibitem[C]{C}A. Cannas da Silva, Lectures on symplectic geometry. Lecture Notes in Mathematics, 1764. Springer-Verlag, Berlin, 2001. xii+217 pp. ISBN: 3-540-42195-5 MR1853077
\bibitem[D]{D}T. Delzant, Hamiltoniens p\'{e}riodiques et images convexes de l'application moment. (French) [Periodic Hamiltonians and convex images of the momentum mapping] Bull. Soc. Math. France 116 (1988), no. 3, 315--339. MR0984900 (90b:58069)
\bibitem[F]{F}Fulton, William, Introduction to toric varieties. Annals of Mathematics Studies, 131. The William H. Roever Lectures in Geometry. Princeton University Press, Princeton, NJ, 1993. xii+157 pp. ISBN: 0-691-00049-2
\bibitem[G]{G}V. Guillemin, Moment maps and combinatorial invariants of Hamiltonian $T\sp n$-spaces. Progress in Mathematics, 122. Birkh\"{a}user Boston, Inc., Boston, MA, 1994. viii+150 pp. ISBN: 0-8176-3770-2
\bibitem[KL]{KL}Y. Karshon and E. Lerman, Non-compact symplectic toric manifolds. arXiv:0907.2891v2
\bibitem[KM]{KM}A. Kriegl, P. Michor, The convenient setting of global analysis. Mathematical Surveys and Monographs, 53. American Mathematical Society, Providence, RI, 1997. x+618 pp. ISBN: 0-8218-0780-3
\bibitem[L1]{L}E.Lerman, Symplectic cuts. Math. Res. Lett. 2 (1995), no. 3, 247--258. MR1338784 (96f:58062)
\bibitem[L2]{L1}E.Lerman, Contact Cuts, Israel J. Math , 124 (2001), 77--92; www.arXiv.org/abs/math.SG/000204
\bibitem[L3]{L2}E.Lerman, Contact toric manifolds, J. Symplectic Geom. 1 (2003), no. 4, 785--828.
\bibitem[L4]{LL}E.Lerman, Homotopy groups of $K$-contact toric manifolds. Trans. Amer. Math. Soc. 356 (2004), no. 10, 4075--4083 (electronic). MR2058839 (2005b:53136)
\bibitem[M]{M}J. Milnor, Remarks on infinite-dimensional Lie groups. Relativity, groups and topology, II (Les Houches, 1983), 1007--1057, North-Holland, Amsterdam, 1984.
\bibitem[MSY]{MSY}D. Martelli, J. Sparks, S.T. Yau, The geometric dual of $a$-maximisation for toric Sasaki-Einstein manifolds. Comm. Math. Phys. 268 (2006), no. 1, 39--65. 
\bibitem[O]{O}H. Omori, Infinite-dimensional Lie groups. Translated from the 1979 Japanese original and revised by the author. Translations of Mathematical Monographs, 158. American Mathematical Society, Providence, RI, 1997. xii+415 pp. ISBN: 0-8218-4575-6
\bibitem[T]{T}Y. Tashiro, On contact structure of hypersurfaces in complex manifolds. II. Tohoku Math. J. (2) 15 1963 167--175.
\bibitem[Y]{Y}K. Yokoyama, The classification of contact toric manifolds, Master's thesis, Tokyo Institute of Technology, 2007.
\end{thebibliography}
\end{document}